\documentclass{article}%
\usepackage{graphicx}
\usepackage{amsmath}
\usepackage{amsfonts}
\usepackage{amssymb}
\usepackage{setspace}
\usepackage[hmargin=3.5cm,vmargin=3.5cm]{geometry}%
\providecommand{\U}[1]{\protect\rule{.1in}{.1in}}
\providecommand{\U}[1]{\protect\rule{.1in}{.1in}}
\providecommand{\U}[1]{\protect\rule{.1in}{.1in}}
\providecommand{\U}[1]{\protect\rule{.1in}{.1in}}
\providecommand{\U}[1]{\protect\rule{.1in}{.1in}}
\providecommand{\U}[1]{\protect\rule{.1in}{.1in}}
\providecommand{\U}[1]{\protect\rule{.1in}{.1in}}
\providecommand{\U}[1]{\protect\rule{.1in}{.1in}}
\providecommand{\U}[1]{\protect\rule{.1in}{.1in}}
\providecommand{\U}[1]{\protect\rule{.1in}{.1in}}
\providecommand{\U}[1]{\protect\rule{.1in}{.1in}}
\providecommand{\U}[1]{\protect\rule{.1in}{.1in}}
\providecommand{\U}[1]{\protect\rule{.1in}{.1in}}
\providecommand{\U}[1]{\protect\rule{.1in}{.1in}}
\providecommand{\U}[1]{\protect\rule{.1in}{.1in}}
\providecommand{\U}[1]{\protect\rule{.1in}{.1in}}
\providecommand{\U}[1]{\protect\rule{.1in}{.1in}}
\providecommand{\U}[1]{\protect\rule{.1in}{.1in}}
\providecommand{\U}[1]{\protect\rule{.1in}{.1in}}
\providecommand{\U}[1]{\protect\rule{.1in}{.1in}}
\providecommand{\U}[1]{\protect\rule{.1in}{.1in}}
\providecommand{\U}[1]{\protect\rule{.1in}{.1in}}
\providecommand{\U}[1]{\protect\rule{.1in}{.1in}}
\providecommand{\U}[1]{\protect\rule{.1in}{.1in}}
\providecommand{\U}[1]{\protect\rule{.1in}{.1in}}
\providecommand{\U}[1]{\protect\rule{.1in}{.1in}}
\providecommand{\U}[1]{\protect\rule{.1in}{.1in}}
\providecommand{\U}[1]{\protect\rule{.1in}{.1in}}

\newtheorem{theorem}{Theorem}
{}

\newtheorem{condition}{Condition}

\newtheorem{corollary}{Corollary}

\newtheorem{definition}{Definition}

\newtheorem{lemma}{Lemma}
{}

\newtheorem{proposition}{Proposition}

\newenvironment{proof}[1][Proof]{\textbf{#1.} }{\ \rule{0.5em}{0.5em}}

\begin{document}

\title{On the Spectral Singularities and Spectrality of the Hill Operator }
\author{O. A. Veliev\\{\small Depart. of Math., Dogus University, Ac\i badem, Kadik\"{o}y, \ }\\{\small Istanbul, Turkey.}\ {\small e-mail: oveliev@dogus.edu.tr}}
\date{}
\maketitle

\begin{abstract}
First we study the spectral singularity at infinity and investigate the
connections of the spectral singularities and the spectrality of the Hill
operator. Then \ we consider the spectral expansion when there is not the
spectral singularity at infinity.



\end{abstract}

\section{\textbf{Introduction}}

In this paper we investigate the one dimensional Schr\"{o}dinger operator
$L(q)$ generated in $L_{2}(-\infty,\infty)$ by the differential expression
\begin{equation}
l(y)=-y^{^{\prime\prime}}(x)+q(x)y(x),
\end{equation}
where $q$ is 1-periodic, Lebesgue integrable on $[0,1]$ and complex-valued
potential. Without loss of generality we assume that the integral of $q$ over
$[0,1]$ is $0.$ It is well-known [1,5-9] that the spectrum $\sigma(L)$ of the
operator $L$ is the union of the spectra $\sigma(L_{t})$ of the operators
$L_{t}(q)$ for $t\in(-\pi,\pi]$ generated in $L_{2}[0,1]$ by the expression
(1) and the boundary conditions
\begin{equation}
y(1)=e^{it}y(0),\text{ }y^{^{\prime}}(1)=e^{it}y^{^{\prime}}(0).
\end{equation}
The eigenvalues of $L_{t}$ are the roots of the characteristic equation
\begin{equation}
F(\lambda)=2\cos t,
\end{equation}
where $F(\lambda)=:\varphi^{^{\prime}}(1,\lambda)+\theta(1,\lambda)$ is the
Hill discriminant $\theta(x,$ $\lambda)$ and $\varphi(x,$ $\lambda)$ are the
solutions of the equation $l(y)=\lambda y$ satisfying the following initial conditions

$\theta(0,\lambda)=\varphi^{^{\prime}}(0,\lambda)=1,$ $\theta^{^{\prime}%
}(0,\lambda)=\varphi(0,\lambda)=0.$

In this paper we study the spectral singularity of $L(q)$ at infinity,
investigate the connections of the spectral singularities and the spectrality
of $L(q)$ and consider the spectral expansion of $L(q)$ when there is not the
spectral singularity at infinity. Note that the spectral singularities of the
operator $L(q)$ are the points of its spectrum in neighborhoods of which the
projections of $L(q)$ are not uniformly bounded (see [9] and [10]). McGarvey
[6] proved that $L(q)$ is a spectral operator if and only if the projections
of the operators $L_{t}(q)$ are bounded uniformly with respect to $t$ in
$(-\pi,\pi]$. Tkachenko proved in [9] that the non-self-adjoint operator $L$
can be reduced to triangular form if all eigenvalues of the operators $L_{t}$
for $t\in(-\pi,\pi]$ are simple. Gesztezy and Tkachenko [3,4] proved two
versions of a criterion for the Hill operator $L(q)$ with $q\in L_{2}[0,1]$ to
be a spectral operator of scalar type, in sense of Danford, one analytic and
one geometric. The analytic version was stated in term of the solutions of
Hill's equation. The geometric version of the criterion uses algebraic and
geometric \ properties of the spectra of periodic/antiperiodic and Dirichlet
boundary value problems. In paper [12, 13] we found the conditions on the
potential $q$ such that $L(q)$ has no spectral singularity at infinity and it
is an asymptotically spectral operator.

Now let us recall the precise definition of the spectral singularities and
asymptotic spectrality of $L(q).$ Following \ [4, 10], we define the
projections and the spectral singularities of $L$ as follows. By Definition
2.4 of [4], a closed arc
\begin{equation}
\gamma=:\{z\in\mathbb{C}:z=\lambda(t),t\in\lbrack\alpha,\beta]\}
\end{equation}
with $\lambda(t)$ continuous on the closed interval $[\alpha,\beta]$, analytic
in an open neighborhood of $[\alpha,\beta]$ and $F(\lambda(t))=2\cos t,$
\[
\text{ }\frac{\partial F(\lambda(t))}{\partial\lambda}\neq0,\text{ }\forall
t\in\lbrack\alpha,\beta],\text{ }\lambda^{^{\prime}}(t)\neq0,\text{ }\forall
t\in(\alpha,\beta)
\]
is called a regular spectral arc of $L(q).$ The projection $P(\gamma)$
corresponding to the regular spectral arc $\gamma$ is defined by
\begin{equation}
P(\gamma)f=\frac{1}{2\pi}%
{\textstyle\int\limits_{\gamma}}
(\Phi_{+}(x,\lambda)F_{-}(\lambda,f)+\Phi_{-}(x,\lambda)F_{+}(\lambda
,f))\frac{\varphi(1,\lambda)}{p(\lambda)}d\lambda,
\end{equation}
where
\[
\Phi_{\pm}(x,\lambda)=:\theta(x,\lambda)+(\varphi(1,\lambda))^{-1}(e^{\pm
it}-\theta(1,\lambda))\varphi(x,\lambda)
\]
are the Floquet solution and
\[
F_{\pm}(\lambda,f)=\int_{\mathbb{R}}f(x)\Phi_{\pm}(x,\lambda)dx,\text{
}p(\lambda)=\sqrt{4-F^{2}(\lambda)}.
\]

\begin{definition}
We say that $\lambda\in\sigma(L(q))$ is a spectral singularity of $L(q)$ if
for all $\varepsilon>0$\ there exists a sequence $\{\gamma_{n}\}$ of the
regular spectral arcs $\gamma_{n}\subset\{z\in\mathbb{C}:\mid z-\lambda
\mid<\varepsilon\}$ such that
\begin{equation}
\lim_{n\rightarrow\infty}\parallel P(\gamma_{n})\parallel=\infty.
\end{equation}

\end{definition}

In the similar way, we defined in [12] the spectral singularity at infinity.

\begin{definition}
We say that the operator $L$ has a spectral singularity at infinity if there
exists a sequence $\{\gamma_{n}\}$ of the regular spectral arcs such that
$d(0,\gamma_{n})\rightarrow\infty$ as $n\rightarrow\infty$ and (6) holds,
where $d(0,\gamma_{n})$ is the distance from the point $(0,0)$ to the arc
$\gamma_{n}.$
\end{definition}

The asymptotic spectrality of the operator $L(q)$ was defined in [12] as
follows. Let $e(t,\gamma)$ be the spectral projection defined by contour
integration of the resolvent of $L_{t}(q)$, where $\gamma\in R$ and $R$ is the
ring consisting of all sets which are the finite union of the half closed
rectangles. In [6] it was proved\ Theorem 3.5 (for the differential operators
of arbitrary order with periodic coefficients rather than for $L(q)$) which
can be written in the form:

$L(q)$\textit{\ is a spectral operator if and only if }%
\begin{equation}
\sup_{\gamma\in R}(ess\sup_{t\in(-\pi,\pi]}\parallel e(t,\gamma)\parallel
)<\infty.
\end{equation}
According to this theorem, in [12] we gave the following definition of the
asymptotic spectrality.

\begin{definition}
The operator $L(q)$ is said to be an asymptotically spectral operator if there
exists a positive constant $C$ such that
\[
\sup_{\gamma\in R(C)}(ess\sup_{t\in(-\pi,\pi]}\parallel e(t,\gamma
)\parallel)<\infty,
\]
where $R(C)$ is the ring consisting of all sets which are the finite union of
the half closed rectangles lying in $\{\lambda\in\mathbb{C}:\mid\lambda
\mid>C\}.$
\end{definition}

In [12] and [13] we obtained the following results about asymptotic
spectrality under the following conditions.

\begin{condition}
\textit{Let }$q\in W_{1}^{p}[0,1]$\textit{, }$q^{(k)}(0)=q^{(k)}(1),$ for
$\,k=0,1,...,s-1$and $q^{(s)}(0)\neq q^{(s)}(1)$\textit{ for some }$s\leq p.$
Suppose that $q_{n}\sim q_{-n},$ $\mid q_{n}\mid>cn^{-s-1}$ and at least one
of the following inequalities%
\[
\operatorname{Re}q_{n}q_{-n}\geq0,\text{ }\mid\operatorname{Im}q_{n}q_{-n}%
\mid\geq\varepsilon\mid q_{n}q_{-n}\mid
\]
hold for some $c>0$ and $\varepsilon>0$ and for large value of $n,$ where
$q_{n}$ is the Fourier coefficient of the potential $q$ and $q_{n}\sim q_{-n}$
means that $q_{n}=O(q_{-n})$ and $q_{-n}=O(q_{n}).$
\end{condition}

\begin{condition}
Suppose that
\begin{equation}
q(x)=ae^{-i2\pi x}+be^{i2\pi x},
\end{equation}%
\begin{equation}
\mid a\mid=\mid b\mid,\text{ }\inf_{q,p\in\mathbb{N}}\{\mid q\alpha
-(2p-1)\mid\}\neq0,
\end{equation}
where $a$ and $b$ are the complex numbers and $\alpha=\pi^{-1}\arg(ab).$
\end{condition}

In [12] we proved that if Condition 1 holds then $L(q)$ is an asymptotically
spectral operator and has no spectral singularity at infinity\textit{. }It was
proven in [13] that the operator $L(q)$ with the potential (8) is an
asymptotically spectral operator and has no spectral singularity at infinity
if and only if (9) holds.

\section{Spectrum, Spectral Singularity and Spectrality}

First, let us discuss the spectrum of $L(q)$ by using some results of [11, 12]
about the uniform with respect to $t$ in $(-\pi,\pi]$ asymptotic formulae for
eigenvalues of the operator $L_{t}(q)$. Note that, the
formula\ $f(k,t)=O(h(k))$ is said to be uniform with respect to $t$ in a set
$I$ if there exist positive constants $M$ and $N,$ independent of $t,$ such
that
\[
\mid f(k,t))\mid<M\mid h(k)\mid
\]
for all $t\in I$ and $\mid k\mid\geq N.$

In the case $q=0$ the eigenvalues and eigenfunctions of $L_{t}(q)$ are $(2\pi
n+t)^{2}$ and $e^{i(2\pi n+t)x}$ for $n\in\mathbb{Z}$. In [11] we proved that
the large eigenvalues of the operators $L_{t}(q)$\ for $t\neq0,\pi$ consist of
the sequence $\left\{  \lambda_{n}(t):\mid n\mid\gg1\right\}  $ satisfying\ \
\begin{equation}
\lambda_{n}(t)=(2\pi n+t)^{2}+O(\frac{\ln\left\vert n\right\vert }{n}).
\end{equation}
This asymptotic formula is uniform with respect to $t$\ in $[\rho,\pi-\rho
],$where $\rho\in(0,\frac{\pi}{2})$. There exists a positive number $N(\rho
),$\ independent of $t,$\ such that the eigenvalues $\lambda_{n}(t)$\ for \ 

$t\in\lbrack\rho,\pi-\rho]$\ and $\mid n\mid>N(\rho)$\ are simple and hence is
analytic function in the neighborhood of $[\rho,\pi-\rho]$.

In the paper [12] \ we proved that there exists a positive integer $N(0)$ such
that the disk
\begin{equation}
U(n,t,\rho)=:\{\lambda\in\mathbb{C}:\left\vert \lambda-(2\pi n+t)^{2}%
\right\vert \leq15\pi n\rho\}
\end{equation}
for $t\in\lbrack0,\rho],$ where $15\pi\rho<1,$ and $n>N(0)$ contains two
eigenvalues (counting with multiplicities) denoted by $\lambda_{n,1}(t)$ and
$\lambda_{n,2}(t)$ and these eigenvalues can be chosen as continuous function
of $t$ on the interval $[0,\rho].$ In addition to these eigenvalues, the
operator $L_{t}(q)$ for $t\in\lbrack0,\rho]$ has only $2N(0)+1$ eigenvalues
denoted by $\lambda_{k}(t)$ for $k=0,\pm1,\pm2,...,\pm N$ (see Remark 2.1 of
[12]). Similarly, there exists a positive integer $N(\pi)$ such that the disk
$U(n,t,\rho)$ for $t\in\lbrack\pi-\rho,\pi]$ and $n>N(\pi)$ contains two
eigenvalues (counting with multiplicities) denoted again by $\lambda_{n,1}(t)$
and $\lambda_{n,2}(t)$ that are continuous function of $t$ on the interval
$[\pi-\rho,\pi].$ In addition to these eigenvalues, the operator $L_{t}(q)$
for $t\in\lbrack\pi-\rho,\pi]$ has only $2N(\pi)+1$ eigenvalues. Thus for $n>$
$N=:\max\left\{  N(\rho),N(0),N(\pi)\right\}  ,$ the eigenvalues
$\lambda_{n,1}(t)$ and $\lambda_{n,2}(t)$ are continuous on $[0,\rho
]\cup\lbrack\pi-\rho,\pi]$ and for$\mid n\mid>N$ the eigenvalue $\lambda
_{n}(t),$ defined by (10), is analytic function in the neighborhood of
$[\rho,\pi-\rho].$ Moreover, by (10) there exist only two eigenvalues
$\lambda_{-n}(\rho)$ and $\lambda_{n}(\rho)$ of the operator $L_{\rho}(q)$
lying in the disk $U(n,\rho,\rho).$ Therefore these 2 eigenvalues coincides
with the eigenvalues $\lambda_{n,1}(\rho)$ and $\lambda_{n,2}(\rho).$ By (10)
$\operatorname{Re}(\lambda_{-n}(\rho))<\operatorname{Re}(\lambda_{n}(\rho)).$
Let $\lambda_{n,2}(\rho)$ be the eigenvalue whose real part is larger. Then%
\begin{equation}
\lambda_{n,1}(\rho)=\lambda_{-n}(\rho),\text{ }\lambda_{n,2}(t)=\lambda
_{n}(\rho),\text{ }\forall n>N
\end{equation}
In the same way we obtain that
\begin{equation}
\lambda_{n,1}(\pi-\rho)=\lambda_{n}(\pi-\rho),\text{ }\lambda_{n,2}(\pi
-\rho)=\lambda_{-(n+1)}(\pi-\rho),\text{ }\forall n>N
\end{equation}
if $\lambda_{n,2}(\pi-\rho))$ is the eigenvalue whose real part is larger. Let
$\Gamma_{-n}$ be the union of the following continuous curves $\left\{
\lambda_{n-1,2}(t):t\in\lbrack\pi-\rho,\pi]\right\}  ,$ $\left\{  \lambda
_{-n}(t):t\in\lbrack\rho,\pi-\rho]\right\}  $ and $\left\{  \lambda
_{n,1}(t):t\in\lbrack0,\rho]\right\}  .$ By (12) and \ (13) these curves are
connected and $\Gamma_{-n}$ is a continuous curve. Similarly, the curve
$\Gamma_{n}$ which is the union of the curves $\left\{  \lambda_{n,2}%
(t):t\in\lbrack0,\rho]\right\}  ,$ $\left\{  \lambda_{n}(t):t\in\lbrack
\rho,\pi-\rho]\right\}  $ and $\left\{  \lambda_{n,1}(t):t\in\lbrack\pi
-\rho,\pi]\right\}  $ is a continuous curve. For $t\in\lbrack0,\rho]$ redenote
$\lambda_{n,1}(t)$ by $\lambda_{-n}(t)$ and $\lambda_{n,2}(t)$ by $\lambda
_{n}(t),$ where $n>N.$ In the same way we put $\lambda_{n}(t)=:\lambda
_{n,1}(t),$ $\lambda_{-(n+1)}(t)=:\lambda_{n,2}(t)$ for $t\in\lbrack\pi
-\rho,\pi]\ $\ and $n>N.$ In this notation we have%
\begin{equation}
\Gamma_{n}=\left\{  \lambda_{n}(t):t\in\lbrack0,\pi]\right\}  .
\end{equation}

The eigenvalues of $L_{-t}(q)$ coincides with the eigenvalues of $L_{t}(q),$
because they are roots of equation (3) and $\cos(-t)=\cos t.$ We define the
eigenvalue $\lambda_{n}(-t)$ of $L_{-t}(q)$ by $\lambda_{n}(-t)=\lambda
_{n}(t)$ for all $t\in(0,\pi).$ Then $\lambda_{n}(t)$ for $\mid n\mid>N$ is an
continuous function on $(-\pi,\pi]$, $\Gamma_{n}=\left\{  \lambda_{n}%
(t):t\in(-\pi,\pi]\right\}  $ and
\[
\sigma(L(q))=%
{\textstyle\bigcup\limits_{t\in(-\pi,\pi]}}
\sigma(L_{t}(q))\supset%
{\textstyle\bigcup\limits_{\mid n\mid>N}}
\Gamma_{n}.
\]
Thus the spectrum $\sigma(L)$ of the operator $L$ contains the continuous
curves $\Gamma_{n}$ for$\mid n\mid>N.$ The remaining part of $\sigma(L)$
consist of finite simple analytic arcs $\gamma_{1},\gamma_{2},...,\gamma_{m}$
whose endpoints are the eigenvalues of $\ L_{0}(q)$ and $L_{\pi}(q)$ and the
roots of the equations $\frac{dF(\lambda)}{d\lambda}=0$ lying in the spectrum
\ (see [9]). On the other hand the above arguments show that $\gamma_{1}%
\cup\gamma_{2}\cup...\cup\gamma_{m}$ is the union of $2N+1$ eigenvalues
(counting multiplicity and denoted by $\lambda_{n}(t)$ for $n=0,\pm1,...,\pm
N)$ of $L_{t}(q)$ for $t\in(-\pi,\pi]$. Moreover if $\lambda_{n}(t)$ is a root
of (3) of multiplicity $k,$ then it is the end point of $k$ curves
$\gamma_{n_{1}},\gamma_{n_{2}},...,\gamma_{n_{k}},$ that is, these $k$ curves
are joined by $\lambda_{n}(t).$ Therefore one can numerate the eigenvalues so
that
\begin{equation}
\gamma_{1}\cup\gamma_{2}\cup...\cup\gamma_{m}=\cup_{\mid n\mid\leq N}%
\Gamma_{n}%
\end{equation}
where $\Gamma_{n}=\left\{  \lambda_{n}(t):t\in(-\pi,\pi]\right\}  $ for $\mid
n\mid\leq N$ are continuous curves. Thus we have
\begin{equation}
\sigma(L(q))=%
{\textstyle\bigcup\limits_{n\in\mathbb{Z}}}
\Gamma_{n}.
\end{equation}

Using (10), (11) and the definition of $\lambda_{n}(t)$ one can readily see
that
\begin{equation}
\left\vert \lambda_{n}(t)-(2\pi k\pm t)^{2}\right\vert \geq\left\vert
(n-k)\right\vert \left\vert n+k\right\vert
\end{equation}
for $k\neq\pm n,\pm(n+1)$ and $t\in(-\pi,\pi],$ where $\left\vert n\right\vert
>N$ . In [11], [12] to write the asymptotic formulas for the eigenfunctions
$\Psi_{n,t}(x)$ corresponding to the eigenvalue $\lambda_{n}(t)$ we used the
following relations
\begin{equation}
(\lambda_{n}(t)-(2\pi k+t)^{2})(\Psi_{n,t},e^{i(2\pi k+t)x})=(q\Psi
_{n,t},e^{i(2\pi k+t)x}),
\end{equation}
\begin{equation}
\left\vert (q\Psi_{n,t},e^{i(2\pi k+t)x})\right\vert <2M
\end{equation}
for $t\in(-\pi,\pi],$ $\left\vert n\right\vert >N$ and $k\in\mathbb{Z},$
where
\[
M=\sup_{n\in\mathbb{Z}}\left\vert q_{n}\right\vert ,\text{ }q_{n}=\int_{0}%
^{1}q(x)e^{-i2\pi nx}dx.
\]
From (17)-(19) we obtain the following, uniform with respect to $t$\ in
$(-\pi,\pi],$ asymptotic formulas
\[
\sum_{k\in\mathbb{Z}\backslash\left\{  \pm n,\pm(n+1)\right\}  }\left\vert
(\Psi_{n,t},e^{i(2\pi k+t)x})\right\vert ^{2}=O(n^{-2})
\]
and \
\[
\sum_{k\in\mathbb{Z}\backslash\left\{  \pm n,\pm(n+1)\right\}  }\left\vert
(\Psi_{n,t},e^{i(2\pi k+t)x})\right\vert =O(\frac{\ln n}{n})
\]
Therefore $\Psi_{n,t}(x)$ has an expansion of the form
\begin{equation}
\Psi_{n,t}(x)=\sum_{k\in\left\{  \pm n,\pm(n+1)\right\}  }u_{n,k}(t)e^{i(2\pi
k+t)x}+h_{n,t}(x),
\end{equation}
where
\begin{equation}
\text{ }\left\Vert h_{n,t}\right\Vert =O(n^{-1}),\text{ }\sup_{x\in
\lbrack0,1],\text{ }t\in(-\pi,\pi]}\mid h_{n,t}(x)\mid=O\left(  \frac
{ln\left\vert n\right\vert }{n}\right)  ,\text{ }(h_{n,t},e^{i(2\pi k+t)x})=0
\end{equation}
for $k\in\left\{  \pm n,\pm(n+1)\right\}  $ and%
\begin{equation}
u_{n,k}(t)=(\Psi_{n,t},e^{i(2\pi k+t)x}).
\end{equation}
Let $\{\chi_{n,t}:n\in\mathbb{Z\}}$ be biorthogonal to $\left\{  \Psi
_{n,t}:n\in\mathbb{Z}\right\}  $ and $\Psi_{n,t}^{\ast}(x)$ be the normalized
eigenfunction of $(L_{t}(q))^{\ast}$ corresponding to $\overline{\lambda
_{n}(t)}.$ Since the boundary condition (2) is self-adjoint we have
$(L_{t}(q))^{\ast}=$ $L_{t}(\overline{q}).$ Therefore, we have
\begin{equation}
\Psi_{n,t}^{\ast}(x)=\sum_{k\in\left\{  \pm n,\pm(n+1)\right\}  }u_{n,k}%
^{\ast}(t)e^{i(2\pi k+t)x}+h_{n,t}^{\ast}(x),
\end{equation}
where $u_{n,k}^{\ast}(t)=(\Psi_{n,t}^{\ast},e^{i(2\pi k+t)x})$ and%
\begin{equation}
\left\Vert h_{n,t}^{\ast}\right\Vert =O(n^{-1}),\text{ }\sup_{x\in
\lbrack0,1],\text{ }t\in(-\pi,\pi]}\mid h_{n,t}(x)\mid=O\left(  \frac
{ln\left\vert n\right\vert }{n}\right)  ,\text{ }(h_{n,t}^{\ast},e^{i(2\pi
k+t)x})=0
\end{equation}
for $k\in\left\{  \pm n,\pm(n+1)\right\}  .$

Introduce the functions
\begin{equation}
\alpha_{n}(t)=(\Psi_{n,t}(x),\Psi_{n,t}^{\ast}(x))_{(0,1)},\text{ }\chi
_{n,t}(x))=\frac{1}{\overline{\alpha_{n}(t)}}\Psi_{n,t}^{\ast}(x),\text{ }%
\end{equation}
where $(.,.)_{(a,b)}$ denotes the inner product in $L_{2}(a,b).$ One can
easily verify that
\begin{equation}
\Psi_{n,t}(x)=\frac{\Phi_{+}(x,\lambda_{n}(t))}{\mid\Phi_{+}(x,\lambda
_{n}(t))\mid},\text{ }\chi_{n,t}(x))=\frac{1}{\overline{\alpha_{n}(t)}}%
\frac{\overline{\Phi_{-}(x,\lambda_{n}(t))}}{\mid\overline{\Phi_{-}%
(x,\lambda_{n}(t))}\mid},
\end{equation}%
\begin{equation}
\Psi_{n,t}(x+1)=e^{it}\Psi_{n,t}(x),\text{ }\Psi_{n,t}^{\ast}(x+1)=e^{it}%
\Psi_{n,t}^{\ast}(x),\text{ }\chi_{n,t}(x+1)=e^{it}\chi_{n,t}(x),
\end{equation}
where $\Phi_{+}$ and $\Phi_{-}$ are defined in (5). Using these formulas and
the equalities%
\begin{equation}
d\lambda=-p(\lambda)\left(  \frac{dF}{d\lambda}\right)  ^{-1}dt,\text{ }%
\frac{dF(\lambda_{n}(t))}{d\lambda}=-\varphi(1,\lambda_{n}(t))(\Phi
_{+}(x,\lambda_{n}(t)),\overline{\Phi_{-}(x,\lambda_{n}(t))})
\end{equation}
and changing the variable $\lambda$ to the variable $t$ in the integral (5) we
get
\begin{equation}
P(\gamma)f=\frac{1}{2\pi}%
{\textstyle\int\limits_{\delta}}
(f,\chi_{n,t})_{\mathbb{R}}\Psi_{n,t}dt,
\end{equation}
where $\delta=\{t\in(-\pi,\pi]:\lambda_{n}(t)\in\gamma\}$ and $\gamma\in
\Gamma_{n}.$

It is very natural that the projection of the operator $L(q)$ is connected
with the projection
\begin{equation}
e(t,C)=:P_{t}(\lambda_{n}(t))f=\frac{1}{2\pi i}%
{\textstyle\int\limits_{C}}
(L_{t}-\lambda I)^{-1}fdt
\end{equation}
of the operator $L_{t}(q),$ where $C$ is a closed contour containing the
eigenvalue $\lambda_{n}(t)$ but no other eigenvalues.

In this paper \ we use the following proposition which was proved in [10].

\begin{proposition}
Let $\gamma\subset\Gamma_{n}$ be regular spectral arc and $\delta=\{t:$
$\lambda_{n}(t)\in\gamma\}.$ Then the operators $P(\gamma)$ and $P_{t}%
(\lambda_{n}(t))$ for $t\in\delta$ are projections and
\begin{equation}
\parallel P_{t}(\lambda_{n}(t))\parallel=\mid\alpha_{n}(t)\mid^{-1},
\end{equation}%
\begin{equation}
\parallel P(\gamma)\parallel=\sup_{t\in\delta}\parallel P_{t}(\lambda
_{n}(t))\parallel=\sup_{t\in\delta}\mid\alpha_{n}(t)\mid^{-1}.
\end{equation}
The eigenvalue $\lambda_{n}(t),$ where $t\in(-\pi,\pi],$ is a spectral
singularity of $L(q)$ if and only if the operator $L_{t}(q)$ has an associated
function corresponding to the eigenvalue $\lambda_{n}(t).$
\end{proposition}

Note that it is well-known and easily checkable that, for $t\neq0,\pi,$
operator $L_{t}$ cannot have two eigenfunctions corresponding to one
eigenvalue $\lambda$ . Indeed if both solution $\varphi(x,\lambda)$ and
$\theta(x,\lambda)$ satisfy the boundary condition (2) then
\[
\varphi^{^{\prime}}(1,\lambda)=e^{it},\text{ }\theta(1,\lambda)=e^{it}%
\]
which contradicts (3) for $t\neq0,\pi.$ Therefore Proposition 1 at once
implies the following

\begin{proposition}
A number $\lambda\in\sigma(L_{t}(q))$ for $t\in(0,\pi)$ is a spectral
singularity of $L(q)$ if and only if it is a multiple eigenvalue of $L_{t}%
(q)$. Moreover, $\Gamma_{n}$ does not contain the spectral singularities if
and only if there exists $\beta>0$ such that $\mid\alpha_{n}(t)\mid^{-1}%
<\beta$ for all $t\in(0,\pi).$
\end{proposition}

Now we are ready to prove the main results of this chapter.

\begin{theorem}
The following statements are equivalent

$(a)$ The operator $L(q)$ has no spectral singularity at infinity.

$(b)$ $L(q)$ is an asymptotically spectral operator.

$(c)$\ There exists $N$ such that the following are satisfied: $(i)$ The
operator $L(q)$ has no spectral singularity on $\Gamma_{n}$ for$\mid n\mid>N$
and hence $L(q)$ may have at most finitely many spectral singularities.
$(ii)$\ For $\left\vert n\right\vert >N$ and $t\in(0,\pi)$ the eigenvalues
$\lambda_{n}(t)$ are simple. $(iii)$ The algebraic multiplicity of
$\lambda_{n}(0)$ and $\lambda_{n}(\pi)$ for $\left\vert n\right\vert >N$ are
equal to their geometric multiplicities, that is, there are not associated
functions corresponding to those eigenvalues. $(iiii)$ There exists a constant
$d$ such that%
\begin{equation}
\mid\alpha_{n}(t)\mid^{-1}<d
\end{equation}
for all $\mid n\mid>N$ and $t\in(-\pi,0)\cup(0,\pi).$
\end{theorem}

\begin{proof}
\ First let us show that $(c)$ implies $(a).$ From (32) and (33) it follows
that  $\parallel P(\gamma)\parallel<d$ for all regular spectral arcs
$\gamma\in\Gamma_{n}$ whenever$\mid n\mid>N$ $.$ Therefore, by Definition 2,
$(a)$ holds.

Now we prove that $(a)$ implies $(c).$ Suppose that $(a)$ holds. If $(i)$ does
not hold then there exist a sequence of pairs $\left\{  (n_{k},t_{k})\right\}
$ such that $\left\vert n_{k}\right\vert \rightarrow\infty$ and $\lambda
_{n_{k}}(t_{k})$ for $k=1,2,...,$ are spectral singulatities of $L(q).$ Then
by Definition 1 one can choose a sequence of regular arc $\gamma_{n_{k}%
}\subset\Gamma_{n_{k}}$ such that
\begin{equation}
\parallel P(\gamma_{n_{k}})\parallel>k
\end{equation}
which contradicts $(a)$ (see Definition 2). To complete the proof \ of $(i)$
it remains to note that the spectral singularities of $L(q)$ are contained in
the set
\begin{equation}
\{\lambda:\frac{dF(\lambda)}{d\lambda}=0,\text{ }\lambda\in\sigma(L(q))\}
\end{equation}
and the entire function $\frac{dF(\lambda)}{d\lambda}$ has at most finite
number of roots on the compact set $\cup_{|n|\leq N}\Gamma_{n}.$

If $(ii)$ does not hold then there exist a sequence of pairs $\left\{
(n_{k},t_{k})\right\}  $ such that $t_{k}\in(0,\pi)$ and $\lambda_{n_{k}%
}(t_{k})$ is a multiple eigenvalue. Then by Proposition 2 the numbers
$\lambda_{n_{k}}(t_{k})$ for $k=1,2,...,$ are spectral singulatities of $L(q)$
which contradicts $(i).$

Similarly, if $(iii)$ does not hold then there exist infinitely many
associated functions and hence by Proposition 1 there exist infinitely many
spectral singulatities which again contradicts $(i).$

If $(iiii)$ does not hold then there exist a sequence of pairs $\left\{
(n_{k},t_{k})\right\}  ,$ where $\left\vert n_{k}\right\vert \rightarrow
\infty$ and $t_{k}\in(-\pi,0)\cup(0,\pi),$ such that%
\[
\left\vert \alpha_{n_{k}}(t_{k})\right\vert ^{-1}>k.
\]
Moreover by $(ii)$ the eigenvalues $\lambda_{n_{k}}(t_{k})$ for $\left\vert
n_{k}\right\vert >N$ are simple. Therefore using the continuity of
$\alpha_{n_{k}}(t)$ at $t_{k}$ and (32) we obtain that there exists a sequence
of regular arc $\gamma_{n_{k}}\subset\Gamma_{n_{k}}$ such that $\parallel
P(\gamma_{n_{k}})\parallel\rightarrow\infty$ which contradicts $(a)$ (see
Definition 2).

Now we prove that $(a)$ and $(b)$ are equivalent. If $(a)$ does not hold then
it is clear that $(b)$ also does not holds. Suppose that $(a)$ holds. Then
(33) holds too. Let $C$ be a positive constant such that if $\lambda_{n}%
(t)\in\{\lambda\in\mathbb{C}:\mid\lambda\mid>C\},$ then $\mid n\mid>N$ for all
$t\in(-\pi,\pi],$ where $N$ is defined in $(c)$. If $\gamma\in R(C),$ then
$\gamma$ encloses finite number \ of the simple eigenvalues of \ $L_{t}(q).$
Thus, there exists a finite subset $J(t,\gamma)$ of $\{n\in\mathbb{Z}$: $\mid
n\mid>N\}$ such that the eigenvalue $\lambda_{k}(t)$ lies inside $\gamma$ if
and only if $k\in J(t,\gamma).$ It is well-known that these eigenvalues are
the simple poles of the Green function of $L_{t}(q)$ and the projection
$e(t,\gamma)$ has the form
\begin{equation}
e(t,\gamma)f=\sum_{n\in J(t,\gamma)}\frac{1}{\alpha_{n}(t)}(f,\Psi_{n,t}%
^{\ast})\Psi_{n,t}.
\end{equation}
Therefore the proof of the theorem follows from the following lemma
\end{proof}

\begin{lemma}
If (33) holds then there exists a positive constant $D$ such that
\begin{equation}
\parallel\sum_{n\in J}\frac{1}{\alpha_{n}(t)}(f,\Psi_{n,t}^{\ast})\Psi
_{n,t}\parallel^{2}<D\parallel f\parallel^{2}%
\end{equation}
for all $t\in(-\pi,0)\cup(0,\pi)$ and for all subset $J$ of $\left\{
n\in\mathbb{Z}:\mid n\mid>N\right\}  .$
\end{lemma}

\begin{proof}
First let us prove that there exists a positive constant $c$ such that%
\begin{equation}
\parallel\sum_{n\in J}\frac{1}{\alpha_{n}(t)}(f,\Psi_{n,t}^{\ast})\Psi
_{n,t}\parallel^{2}<c\sum_{n:\mid n\mid>N}\mid(f,\Psi_{n,t}^{\ast})\mid^{2}%
\end{equation}
for all $t\in(-\pi,0)\cup(0,\pi).$ By (20) and (21) we have%
\begin{equation}
\parallel\sum_{n\in J}\frac{1}{\alpha_{n}(t)}(f,\Psi_{n,t}^{\ast})\Psi
_{n,t}\parallel^{2}=S_{1}+S_{2}^{2}%
\end{equation}
where
\[
S_{1}=\parallel\sum_{n\in J}\frac{1}{\alpha_{n}(t)}(f,\Psi_{n,t}^{\ast}%
)(\sum_{k\in\left\{  \pm n,\pm(n+1)\right\}  }u_{n,k}(t)e^{i(2\pi
k+t)x})\parallel^{2},
\]%
\[
S_{2}=\parallel\sum_{n\in J}\frac{1}{\alpha_{n}(t)}(f,\Psi_{n,t}^{\ast
})h_{n,t}\parallel
\]
Since $\left\{  e^{i(2\pi n+t)x}:n\in\mathbb{Z}\right\}  $ is an orthonormal
basis and $\left\vert u_{n,k}(t)\right\vert \leq1$ (see (22)) using (33) and
the Bessel inequality one can easily verify that
\begin{equation}
S_{1}\leq16d^{2}\sum_{n:\mid n\mid>N}\mid(f,\Psi_{n,t}^{\ast})\mid^{2}%
\end{equation}
It follows from (33) and (21) that
\[
S_{2}<c_{1}\sum_{n:\mid n\mid>N}\mid(f,\Psi_{n,t}^{\ast})\mid\frac{1}{\mid
n\mid}%
\]
for some constant $c_{1}.$ Now using the Schwarz inequality for $l_{2}$ we
obtain%
\[
S_{2}<c_{1}\left(  \sum_{n:\mid n\mid>N}\mid(f,\Psi_{n,t}^{\ast})\mid
^{2}\right)  ^{1/2}\left(  \sum_{n:\mid n\mid>N}\frac{1}{n^{2}}\right)
^{1/2}\leq\frac{c_{1}}{\sqrt{N}}\left(  \sum_{n:\mid n\mid>N}\mid(f,\Psi
_{n,t}^{\ast})\mid^{2}\right)  ^{1/2}.
\]
Thus (38) follows from (39) and (40). Therefore to prove (37) it is enough
show that there exists a positive constant $c_{2}$ such that
\begin{equation}
\sum_{n:\mid n\mid>N}\mid(f,\Psi_{n,t}^{\ast})\mid^{2}\leq c_{2}\parallel
f\parallel^{2}.
\end{equation}
It can be proved arguing as in the proof of (38) and using (23) instead of (20)
\end{proof}

\begin{theorem}
The operator $L(q)$ is a spectral operator if and only if it has no spectral
singularities at $\sigma(L(q))$ and \ at infinity.
\end{theorem}

\begin{proof}
If $L(q)$ is a spectral operator, then by (7), (30) and Proposition 1 there
exists a positive constant $c_{3}$ such that $\parallel P(\gamma
)\parallel<c_{3}$ for all regular spectral arcs $\gamma.$ Therefore, by
Definition 1 and Definition 2, the operator $L(q)$ has no spectral
singularities at $\sigma(L(q))$ and \ at infinity. If $L(q)$ has no spectral
singularities at $\sigma(L(q))$ and \ at infinity then (33) holds for all
$n\in\mathbb{Z}$ and $t\in(-\pi,0)\cup(0,\pi).$ Using this and arguing as in
the proof of the implication $(a)\Longrightarrow(b)$ we obtain that $L(q)$ is
a spectral operator
\end{proof}

\begin{definition}
The component $\Gamma_{n}$ ,defined by (14), of the spectrum $\sigma(L(q))$ of
the operator $L(q)$ is said to be separated if $\Gamma_{n}\cap\Gamma
_{m}=\emptyset$ for all $m\neq n.$ Thus all component $\Gamma_{n}$ of the
spectrum $\sigma(L(q))$ are separated if and only if \ all eigenvalues of the
operators $L_{t}$ for $t\in(-\pi,\pi]$ are simple.
\end{definition}

\begin{corollary}
The Mathieu operator $L(2a\cos x),$ where $a$ is a complex number, is a
spectral operator if and only if
\begin{equation}
\inf_{q,p\in\mathbb{N}}\{\mid2q\alpha-(2p-1)\mid\}\neq0,
\end{equation}
\ and all eigenvalues of the operators $L_{t}$ for $t\in(-\pi,\pi]$ are simple.
\end{corollary}

\begin{proof}
Using the result of [13] mentioned in the end of introduction which states
that $L(2a\cos x)$ has no spectral singularities\ at infinity if and only if
(42) holds, we prove the corollary as follows. If (42) holds and all
eigenvalues are simple then $L$ has no spectral singularity at infinity and in
spectrum and hence, by Theorem 2, $L$ is a spectral operator. Now suppose that
$L$ is a spectral operator. Then, by Theorem 2, $L$ has no spectral
singularities\ at infinity and hence (42) holds. It remains to prove that all
eigenvalues are simple. By Proposition 2 the eigenvalues $\lambda_{n}(t)$ for
$t\in(0,\pi)$ and for all $n$ are simple, since $L$ has no spectral
singularity in spectrum (see Theorem 2). Now to complete the proof of the
corollary it remains to prove that the eigenvalues $\lambda_{n}(0)$ and
$\lambda_{n}(\pi)$ for $n\in\mathbb{Z}$ are simple. By Proposition 1 the
operators $L_{0}$ and $L_{\pi}$ have no associated functions. On the other
hand, it is well-known that, the geometric multiplicity of $\lambda_{n}(0)$
and $\lambda_{n}(\pi)$ for all $n$ is $1$ (see p. 34-35 of [1]. Note that in
[1] this theorem was proved for the real $a$. However, the proof pass through
for the complex $a$ without any change). Therefore $\lambda_{n}(0)$ and
$\lambda_{n}(\pi)$ for all $n$ are simple
\end{proof}

\section{On the Spectral Expansion of the Asymptotically Spectral Hill
Operators}

In this section we examine the spectral expansion theorem in the case when
$L(q)$ has no spectral singularity at infinite by using Section 2 and the
results of [2] where it was proved that every function $f\in L_{2}%
(-\infty,\infty)$ can be represented in the form
\begin{equation}
f(x)=\frac{1}{2\pi}\int\limits_{0}^{2\pi}f_{t}(x)dt,
\end{equation}
where
\begin{equation}
f_{t}(x)=\sum\limits_{k=-\infty}^{\infty}f(x+k)e^{ikt},\text{ }\int_{-\infty
}^{\infty}\left\vert f(x)\right\vert ^{2}dx=\frac{1}{2\pi}\int\limits_{0}%
^{2\pi}\int\limits_{0}^{1}\left\vert f_{t}(x)\right\vert ^{2}dxdt
\end{equation}
and the followings hold%
\begin{equation}
f_{t}(x+1)=e^{it}f_{t}(x),\text{ }(f,\chi_{k,t})_{(-\infty,\infty).}%
=(f_{t},\chi_{k,t})_{(0,1)}=\frac{1}{\alpha_{k}(t)}(f_{t},\Psi_{k,t}^{\ast
})=:a_{k}(t).
\end{equation}
In this case, by Theorem 1, the root of $\frac{dF(\lambda)}{d\lambda}=0$ lying
in the set $\left\{  \lambda_{n}(t):t\in(0,\pi),\text{ }n\in\mathbb{Z}%
\right\}  $ is finite. Denote these roots (if exists) by $\lambda_{1}%
,\lambda_{2},...\lambda_{m}$. Let $t_{1},t_{2},...t_{m}$ be a point of
$(0,\pi)$ such that $\lambda_{k}\in\sigma(L_{t_{k}})$. Introduce the set
$E=(-\pi,\pi)\backslash\left\{  0,\pm t_{1},\pm t_{2},...\pm t_{m}\right\}  .$
By the definition of $E$, if $t\in E,$ then the eigenvalues $\lambda_{k}(t)$,
for all $k\in\mathbb{Z},$ are simple and the system $\left\{  \Psi
_{k,t}(x):k\in\mathbb{Z}\right\}  $ of eigenfunctions of $L_{t}$ forms a Riesz
basis in $L_{2}[0,1]$, since if $t\neq0,\pi$ then the system of the
eigenfunctions and associated functions of $L_{t}(q)$ with potential $q\in
L_{1}[0,1]$ forms, Riesz basis of $L_{2}[0,1]$ (see [11]). Therefore
\begin{equation}
f_{t}(x)=\sum\limits_{k=-\infty}^{\infty}a_{k}(t)\Psi_{k,t}(x),
\end{equation}
where the series converges in the norm of $L_{2}(0,1).$ This with (43) implies
that
\begin{equation}
f(x)=\frac{1}{2\pi}\int\limits_{-\pi}^{\pi}f_{t}(x)dt=\frac{1}{2\pi}%
\int\limits_{E}f_{t}(x)dt=\frac{1}{2\pi}\int\limits_{E}\sum\limits_{k=-\infty
}^{\infty}a_{k}(t)\Psi_{k,t}(x)dt
\end{equation}

Since in this section it is assumed that $L(q)$ has no spectral singularity at
infinite, by Theorem 1 there are at most finite number of spectral
singularities denoted by $\lambda_{1},\lambda_{2},...\lambda_{s},$ where
$s\geq m,$ since by Proposition 2, $\lambda_{1},\lambda_{2},...\lambda_{m}$.
are spectral singularities. Let $S$ \ be the set of integers such that
$\Gamma_{n}$ contains spectral singularities for $n\in S.$ Note that $S$ is a
finite subset of $\mathbb{Z}.$ By Proposition 2 and (45), (44) if $n\notin S$
then%
\begin{equation}
\int\limits_{0}^{2\pi}\int\limits_{0}^{1}\left\vert a_{k}(t)\Psi
_{k,t}(x)\right\vert ^{2}dxdt\leq\beta^{2}\int\limits_{0}^{2\pi}%
\int\limits_{0}^{1}\left\vert f_{t}(x)\right\vert ^{2}dxdt<\infty
\end{equation}
Therefore by Fubini theorem
\begin{equation}
\int\limits_{E}a_{k}(t)\Psi_{k,t}(x)dt
\end{equation}
exists for almost all $x$ when $k\notin S$. Similarly by (33) the integral
(49) exists for $\left\vert k\right\vert >N$.

\begin{theorem}
If the operator $L(q)$ has no spectral singularity at infinity, then every
function $f\in L_{2}(-\infty,\infty)$ has the spectral decomposition
\begin{equation}
f(x)=\frac{1}{2\pi}\int\limits_{E}\sum_{k\in S}a_{k}(t)\Psi_{k,t}%
(x)dt+\sum_{k\in\mathbb{Z}\backslash S}\frac{1}{2\pi}\int\limits_{E}%
a_{k}(t)\Psi_{k,t}(x)dt,
\end{equation}
where the series in (50) converges in the norm of $L_{2}(a,b)$ for every
$a,b\in\mathbb{R}.$
\end{theorem}

\begin{proof}
First let us consider the series
\begin{equation}
\sum_{k>N}a_{k}(t)\Psi_{k,t}(x),
\end{equation}
where $N$ is defined in Theorem 1 and $t\in E$. Let $R_{n}(x,t)$ be remainder
of (51)%
\[
R_{n}(x,t)=\sum_{k>n}a_{k}(t)\Psi_{k,t}(x),
\]
where $n>N.$ Since the series (51) converges in the norm of $L_{2}(0,1)$ by
(45) and (27) we have
\begin{equation}
R_{n}(x+1,t)=e^{it}R_{n}(x,t),\text{ }R_{n}(.,t)\in L_{2}(-m,m)
\end{equation}
for $t\in E$ and for all $m\in\mathbb{N}.$ Repearting the proof of (38) and
using (52) we obtain
\begin{equation}
\parallel R_{n}(.,t)\parallel_{(-m,m)}^{2}\leq2mc_{4}\sum_{k:\mid k\mid>n}%
\mid(f_{t},\Psi_{k,t}^{\ast})_{(0,1)}\mid^{2}%
\end{equation}
for some constant $c_{4},$ where $\parallel f\parallel_{(-m,m)}$ is the
$L_{2}(-m,m)$ norm of $f$. On the other hand, it follows from (33), (23) and
(24) that
\begin{equation}
\sum_{k:\mid k\mid>n}\mid(f_{t},\Psi_{k,t}^{\ast})_{(0,1)}\mid^{2}\leq
c_{5}\sum_{k:\mid k\mid>n}\mid(f_{t},e^{i(2\pi k+t)x})_{(0,1)}\mid
^{2}+c\parallel f_{t}\parallel_{(0,1)}^{2}n^{-1}%
\end{equation}
for some constant $c_{5}.$ Now using the Parseval equality for $L(q)$ when
$q=0$ (see [2]) we obtain
\[
\sum_{k\in\mathbb{Z}}\frac{1}{2\pi}\int\limits_{E}\mid(f_{t},e^{i(2\pi
k+t)x})_{(0,1)}\mid^{2}=\int_{-\infty}^{\infty}\left\vert f(x)\right\vert
^{2}dx=\frac{1}{2\pi}\int\limits_{E}\parallel f_{t}\parallel_{(0,1)}^{2}%
\]
Therefore by (54) and (53) the following integral exists and%
\begin{equation}
I_{n}=:\int\limits_{E}\int\limits_{(-m,m)}\mid R_{n}(x,t)\mid^{2}%
dxdt\rightarrow0
\end{equation}
as $n\rightarrow\infty.$ Thus by Fubini theorem $R_{n}(x,t)$ is integrable
with respect to $t$ for almost all $x.$

Now using the inequality
\[
\left\vert \int_{E}f(t)dt\right\vert ^{2}\leq2\pi\int_{E}\left\vert
f(t)\right\vert ^{2}dt,
\]
Fubini theorem and then (55), we obtain%
\begin{equation}
\parallel\int\limits_{E}\sum_{k>n}a_{k}(t)\Psi_{k,t}dt\parallel_{(-m,m)}%
^{2}\leq2\pi\int\limits_{(-m,m)}\int\limits_{E}\mid\sum_{k>n}a_{k}%
(t)\Psi_{k,t}(x)\mid^{2}dtdx=I_{n}\rightarrow0
\end{equation}
as $n\rightarrow\infty$. Hence the series (51) is integrable and
\begin{equation}
\int\limits_{E}(\sum_{k>N}a_{k}(t)\Psi_{k,t}(x))dt=\sum_{k>N}\int
\limits_{E}a_{k}(t)\Psi_{k,t}(x)dt,
\end{equation}
where the last series converges in the norm of $L_{2}(-m,m)$ for every
$m\in\mathbb{N}.$ In the same way we prove that
\begin{equation}
\int\limits_{E}(\sum_{k<-N}a_{k}(t)\Psi_{k,t}(x))dt=\sum_{k<-N}\int
\limits_{E}a_{k}(t)\Psi_{k,t}(x)dt.
\end{equation}
\ Therefore using (57), (58) and (47) and taking into account that the
integral in (49) exists for $k\notin S$ we get the proof of the theorem
\end{proof}

Now changing the variable to $\lambda$ in (50) by using (28) and taking into
account that $\lambda_{n}(-t)=\lambda_{n}(t)$ we obtain

\begin{theorem}
If the operator $L(q)$ has no spectral singularity at infinity, then every
function $f\in L_{2}(-\infty,\infty)$ has the spectral decomposition
\begin{equation}
f(x)=\frac{1}{\pi}\int\limits_{\gamma(s)}(\phi(x,\lambda))\frac{1}{p(\lambda
)}d\lambda+\frac{1}{\pi}\sum_{k\in\mathbb{Z}\backslash S}\left(
\int\limits_{\Gamma_{k}}(\phi(x,\lambda))\frac{1}{p(\lambda)}d\lambda\right)
,
\end{equation}
where
\begin{align*}
\phi(x,\lambda)  &  =\theta^{^{\prime}}(1,\lambda)h(\lambda)\varphi
(x,\lambda)+\frac{1}{2}(\theta(1,\lambda)-\varphi^{^{\prime}}(1,\lambda
))(h(\lambda)\theta(x,\lambda)+g(\lambda)\varphi(x,\lambda))-\\
&  \ -\varphi(1,\lambda)g(\lambda)\theta(x,\lambda),
\end{align*}%
\[
h(\lambda)=\int\limits_{-\infty}^{\infty}\varphi(x,\lambda)f(x)dx,\text{
}g(\lambda)=\int\limits_{-\infty}^{\infty}\theta(x,\lambda)f(x)dx\text{,
}p(\lambda)=\sqrt{4-F^{2}(\lambda)},
\]
$\gamma(s)=:\bigcup_{k\in S}\Gamma_{k}$ is the part of the spectrum that
contains the spectral singularities and the series in (59) converges in the
norm of $L_{2}(a,b)$ for every $a,b\in\mathbb{R}.$
\end{theorem}

The results of [12] and [13] mentioned in the end of the introduction with
Theorem 3 and Theorem 4 imply

\begin{corollary}
\textit{If the potential }$q$ satisfies one of the Condition 1 and Condition 2
then the spectral expansions of $L(q)$ have the forms (50) \textit{and (59).}
\end{corollary}


\begin{thebibliography}{99}                                                                                               %


\bibitem {}M. S. P. Eastham, The Spectral Theory of Periodic Differential
Equations,\ Scottish Academic Press, Edinburgh, 1973.

\bibitem {}I. M. Gelfand, Expansions in series of eigenfunctions of an
equation with periodic coefficients, Soviet Math. Dokl. 73 (1950), 1117-1120.

\bibitem {}F. Gesztesy and V. Tkachenko, When is a non-self-adjoint Hill
operator a spectral operator of scalar type, C. R. Acad. Sci. Paris, Ser. I,
343 (2006) 239-242.

\bibitem {}F. Gesztesy and V. Tkachenko, A criterion for Hill operators to be
spectral operators of scalar type, J. Analyse Math. 107 (2009) 287--353.

\bibitem {}D. McGarvey, Operators commuting with translations by one. Part I.
Representation theorems, J.Math. Anal. Appl. 4 (1962) 366--410.

\bibitem {}D. McGarvey, Operators commuting with translations by one. Part II.
Differential operators with periodic coefficients in $L_{p}(-\infty,\infty)$,
J. Math. Anal. Appl. 11 (1965) 564--596.

\bibitem {}D. McGarvey, Operators commuting with translations by one. Part
III. Perturbation results for periodic differential operators, J. Math. Anal.
Appl. 12 (1965) 187--234.

\bibitem {}F. S. Rofe-Beketov, The spectrum of nonselfadjoint differential
operators with periodic coefficients, Soviet Math. Dokl. 4 (1963) 1563-1566.

\bibitem {}V. A. Tkachenko, Spectral analysis of nonselfadjoint Schrodinger
operator with a periodic complex potential, Soviet Math. Dokl. 5 (1964), 413-415.

\bibitem {}O. A. Veliev, The spectrum and spectral singularities of
differential operators with complex-valued periodic coefficients, Differential
cprime nye Uravneniya 19 (1983), 1316-1324.

\bibitem {}O. A .Veliev, \ M. Toppamuk Duman, The spectral expansion for a
nonself-adjoint Hill operators with a locally integrable potential, J. Math.
Anal. Appl. 265, (2002) 76-90.

\bibitem {}O. A. Veliev, Asymptotic Analysis of Non-self-adjoint Hill
Operators, Cent. Eur. J. Math., Volume 11, Issue 12 (2013) pp 2234-2256.

\bibitem {}O. A. Veliev, Spectral Analysis of the Non-self-adjoint
Mathieu-Hill Operator, 2013, arXiv:1202.4735v4
\end{thebibliography}
\end{document}